\documentclass[12pt]{article} % use larger type; default would be 10pt

\usepackage[utf8]{inputenc} % set input encoding (not needed with XeLaTeX)

\usepackage{geometry} % to change the page dimensions
\usepackage{amsthm, amssymb}
\usepackage{graphicx} % support the \includegraphics command and options

\usepackage{booktabs} % for much better looking tables
\usepackage{array} % for better arrays (eg matrices) in maths
\usepackage{paralist} % very flexible & customisable lists (eg. enumerate/itemize, etc.)
\usepackage{verbatim} % adds environment for commenting out blocks of text & for better verbatim
\usepackage{subfig} % make it possible to include more than one captioned figure/table in a single float
\usepackage{fancyhdr} % This should be set AFTER setting up the page geometry
\usepackage{sectsty}
\allsectionsfont{\sffamily\mdseries\upshape} % (See the fntguide.pdf for font help)
\newtheorem{lemma}{Lemma}
\newtheorem{theorem}{Theorem}
\newtheorem{statement}{Statement}

\usepackage[nottoc,notlof,notlot]{tocbibind} % Put the bibliography in the ToC
\usepackage[titles,subfigure]{tocloft} % Alter the style of the Table of Contents

\title{Continued fractions and conformal mappings for domains with angle points}
\author{Pyotr N. Ivanshin}
\date{} % Activate to display a given date or no date (if empty),
\begin{document}
\maketitle

{\bf Abstract.} Here we construct the  conformal mappings with the help of the continued fraction approximations. These approximations converge to the algebraic root functions $\sqrt[N]{z}$, $N \in \mathbb{N}$, $z \in \mathbb{C}$, $\mathrm{Re}\, z>0$. We estimate  the convergence rate of the approximation sequences. Also we give the examples that illustrate the  conformal mapping construction.



\section{Introduction}
This article extends and develops  paper \cite{jca}. There we presented the reparametrization method of conformal mapping of the unit disk onto the given simply connected domain with a smooth boundary.  This method   is based on  reduction of Fredholm integral equation to a sufficiently large linear equation system and on the boundary curve reparametrization.   The solution possesses polynomial form that can be easily analyzed.

The method  can be considered as one of the rapidly converging methods according to classification of \cite{Port}. The computation cost is actually similar to Theodorsen’s method or Fornberg method \cite{Forn1}. Let us compare the reparametrization method of \cite{jca} with the other conformal mapping methods.

We do not consider the auxiliary mapping of the unit disk into subdomain of the given domain $D$ as in the set of osculation methods \cite{Ahl}. 
The method of  \cite{jca} does not require a  sufficiently good initial  approximation of the conformal mapping as the graphical methods such as that of \cite{HA}.
The
method does not apply any auxiliary constructions at the domain interior (domain triangulation \cite{Dr2}, circle packing \cite{RS},
domain decompositions, such as meshes of \cite{ZG}). We do not need any iterative conformal mappings as in the zipper algorithm or the Schwartz-Christoffel mapping \cite{Dr1, Lut}. We construct our polynomial solution differently to the Fornberg polynomial method \cite{Forn} that involves consequent approximations through suitable point choice at the domain boundary. Also we do not apply the solutions of auxiliary boundary value
problems (the conjugate function method, Wegmann method \cite{Weg1, Weg2}).
Finally, the advantages of the method presented in \cite{jca} are the following: 1) it is devoid of auxiliary constructions,
2) it brings us to the mapping function  in a polynomial form. The mapping function is a Taylor polynomial for the unit disk or a Laurent polynomial for the annulus in the case of multiconnected domains \cite{ShAb}, \cite{ASI}. 

Let us recall the basic construction steps of the reparametrization method  \cite{jca}. 
 %The construction   is based on solution of the  equation
%\begin{equation}\label{1steqm}
%-\ln \frac{1}{|z(t)|}+i (\theta(t)-\mathrm{arg} z(t))=
%\frac{1}{\pi i} \int\limits_{0}^{2 \pi} \frac{-\ln\frac{1}{|z(\tau)|} +i (\theta(\tau)-\mathrm{arg} z(\tau))}{z(\tau)-z(t)} z'(\tau) d \tau,
%\end{equation}
%here $z(t)$ is the given  boundary of the domain $D$, $\theta(t)$ is the dependence of the polar angle  $\theta$ at the unit circle on the initial parameter $t$ on the contour  $L$ bounding the domain  $D$.

Consider a finite simply connected domain  $D$ bounded by the smooth curve $L= \{z=z(t), \ t\in[0,2\pi]\},\quad z(0)=z(2\pi)$. We trace the domain $D$ counterclockwise along  $L$ as the parameter $t$ increases. We only deal with the cases in which the boundary $L$   representation is as follows:
\begin{equation}\label{7steqm}
z(t)=x(t)+i y(t)=\sum\limits_{k=-m}^{n}c_k e^{i k t}.
\end{equation}
 Note  that any smooth boundary may be approximated by a Fourier polynomial of this type. 
 
If Fourier polynomial representation (\ref{7steqm})  of the curve  $L$ possesses no summands with the negative degrees of  $e^{i  t}$ then the function that maps the unit disk to the domain $D$ is immediately polynomial:
$$
Z(\zeta)=\sum\limits_{k=0}^{n}c_k \zeta^ k.
$$

Assume now that  representation (\ref{7steqm}) contains nonzero coefficients  $c_{-l}$, $l\in \mathbb{N}$. Then it is possible to construct an approximate conformal mapping under reparametrization of (\ref{7steqm})  \cite {Shir} leading to  the coefficients $c_{-l}$, $l\in \mathbb{N}$, elimination.

In order to find this reparametrization $t(\theta)$, $\theta \in [0,2 \pi]$, we first construct the inverse function  
$ \theta(t) = \arg(\zeta(z))|_{z=z(t)\in L}$. Here $\zeta(z)$ is the analytic function that gives the conformal mapping of $D$ onto the unit disk so that  $\zeta(0)=0$.
Let us denote by $q(t)$ the difference $\theta(t)- \arg z(t)$. 
  The necessary  condition for the function $\ln\frac{\zeta(z)}{z}$ to be analytic in $D$ is just as in \cite{Shir} the equation
\begin{eqnarray}\label{2steqm}
q(t)=\frac{1}{\pi}\int\limits_0^{2\pi} q(\tau) (\arg [z(\tau)-z(t)])'_{\tau} d\tau+\frac{1}{\pi}\int\limits_0^{2\pi}\ln |z(\tau)| (\ln |z(\tau)-z(t)|)'_{\tau} d\tau.
\end{eqnarray}

We consider the factor $(e^{i\tau}-e^{ i t})$ in the expression of $z(\tau)-z(t)$ in order to separate the improper VP integral in the last integral equation. Finally, the function $q(t)$
is the solution of the Fredholm integral equation of the second kind (\ref{2steqm}). 
%
%
%\begin{eqnarray}\label{2steqm}
%q(t)=  \frac{1}{\pi } \int\limits_{0}^{2 \pi}q(\tau) \frac{ \partial(\mathrm{arg}(z(\tau)-z(t)))}{\partial \tau} d \tau+\nonumber \\
%+\frac{1}{\pi } \int\limits_{0}^{2 \pi} \ln |z(\tau)|  \frac{\partial [\ln |z(\tau)-z(t)|]}{\partial \tau}d\tau.
%\end{eqnarray}
Note that the integral kernel $\frac1\pi \frac{ \partial(\mathrm{arg}(z(\tau)-z(t)))}{\partial \tau}=\frac1\pi \mathrm{Im}[\frac{ \partial z'(\tau)}{z(\tau)-z(t)}]$ coincides with the operator $K_1(\tau, t)$ of  \cite{Weg3}.  We solve this integral equation reducing it to the finite linear equation system without converging iterations proposed in \cite{Weg3}. 

This integral equation has the set of solutions $q(t)=\alpha_0+\sum\limits_{p=1}^{\infty} \alpha_n \cos n t+\beta_n \sin n t$ that differ by an arbitrary summand $\alpha_0$. Equation (\ref{2steqm}) is uniquely resolvable if we set the value $\frac{1}{2\pi} \int\limits_{0}^{2\pi} q(\tau) d\tau=0$ \cite{jca}, or fix the boundary image $q(t_0)=q_0$ \cite{Weg3}. Indeed, the number $1$ is the simple eigenvalue of $K_1(\tau, t)$ with the eigenfunction $f_0\equiv 1$  \cite{Weg3} and corresponds to rotation of the unit disk. So we search only for the coefficients $\alpha_n, \beta_n$, $n \geq 1$. The operator $K_1(\tau, t)-I$  is invertible  in the subspace of $L^2[0, 2\pi]$ spanned by $\cos n t, \sin n t$, $n \geq 1$.

We search for the approximate function $q(t)$ in the form $q(t)=\sum\limits_{p=1}^{M} \alpha_n \cos n t+\beta_n \sin n t$, $M \in \mathbb{N}$.

Reduce equation (\ref{2steqm}) to the uniquely resolvable finite linear system over the Fourier coefficients  $\alpha_n, \beta_n$, $n \geq 1$, of the function $q(t)$. 

$$
\left( \begin{array}{cc}AA & AB \\ BA & BB \end{array} \right)
\left( \begin{array}{c} \alpha \\ \beta \end{array} \right)  =
\left( \begin{array}{c} F  \\  G \end{array} \right),
$$
here
$$
\alpha = \left( \begin{array}{c}  \alpha_1\\...\\ \alpha_M  \end{array} \right), \quad
\beta = \left( \begin{array}{c} \beta_1\\ ...\\ \beta_M  \end{array} \right). 
$$

The vectors
$$
F= \left( \begin{array}{c} f_1\\...\\ f_M  \end{array} \right),
\quad G= \left( \begin{array}{c} g_1\\...\\ g_M  \end{array} \right) 
$$ 
on the right-hand side of the equation system  consist of the corresponding Fourier coefficients obtained by the following technique: We separate  the summand $\cot\frac{\tau-t}{2}$ in the kernal, apply Hilbert formula  and find Fourier coefficients of the integral with the remained continuous kernel as usual.
%$$
%f_l=(-\sum\limits_{j}\mathrm{Im} z_j^l+\sum\limits_{k} \mathrm{Im} \hat{z}_k^{-l}) \frac{1}{l} + \frac{1}{\pi^2} \int\limits_{0}^{2\pi} \ln |z(\tau)| d\tau \int\limits_{0}^{2\pi} L(\tau,t) \cos (l t) dt,
%$$
%$$
%g_l=(\sum\limits_{j}\mathrm{Re} z_j^l+\sum\limits_{k} \mathrm{Re} \hat{z}_k^{-l}) \frac{1}{l} + \frac{1}{\pi^2} \int\limits_{0}^{2\pi} \ln |z(\tau)| d\tau \int\limits_{0}^{2\pi} L(\tau,t) \sin (l t) dt, l=1,...,M.
%$$

The block matrices of size $M$
$$
AA, AB, BA, BB
$$
consist of the elements
$$
AA=\left( \delta_{l n}-\frac{1}{\pi^2}\int\limits_{0}^{2\pi} \cos (n\tau)  d\tau \int\limits_{0}^{2\pi} K(\tau,t) \cos (l t) dt \right)_{l,n=1}^{M}, 
$$ 
$$
AB=\left( -\frac{1}{\pi^2}\int\limits_{0}^{2\pi} \sin (n\tau)  d\tau \int\limits_{0}^{2\pi} K(\tau,t) \cos (l t) dt \right)_{l,n=1}^{M}, 
$$ 
$$
BA=\left( -\frac{1}{\pi^2}\int\limits_{0}^{2\pi} \cos (n\tau)  d\tau \int\limits_{0}^{2\pi} K(\tau,t) \sin (l t) dt \right)_{l,n=1}^{M}, 
$$ 
$$
BB=\left( \delta_{l n}-\frac{1}{\pi^2}\int\limits_{0}^{2\pi} \sin (n\tau) d\tau \int\limits_{0}^{2\pi} K(\tau,t) \sin (l t) dt \right)_{l,n=1}^{M}, 
$$ 
here $\delta_{l n}$ is the Kronecker delta function.  Also
$$
 K(\tau,t)=\mathrm{Im} \left[\ln\frac{z(\tau)-z(t)}{e^{i\tau}-e^{i t}}\right]'_{\tau}=
 $$
 $$
 =\mathrm{Im} \left[\ln\left(\sum\limits_{k=1}^{n} c_k e^{i k t}\sum\limits_{l=0}^{k-1}e^{i l (\tau-t)}-\sum\limits_{j=1}^{m} c_{-j} e^{-i j \tau}\sum\limits_{l=0}^{j-1}e^{i l (\tau-t)}\right)\right]'_{\tau}
$$
and
$$
L(\tau,t)=\mathrm{Re}\left[\ln\frac{z(\tau)-z(t)}{e^{i\tau}-e^{i t}}\right]'_{\tau}=
$$
$$
=\mathrm{Re} \left[\ln\left(\sum\limits_{k=1}^{n} c_k e^{i k t}\sum\limits_{l=0}^{k-1}e^{i l (\tau-t)}-\sum\limits_{j=1}^{m} c_{-j} e^{-i j \tau}\sum\limits_{l=0}^{j-1}e^{i l (\tau-t)}\right)\right]'_{\tau}.
$$ 

The cost of the linear system solution method is $O(N^2)$, where $N$ is the degree of the Fourier polynomial approximating $q(t)$. Now $z(t)=z(t(\theta))=Z(e^{i \theta})$, here $\theta(t)=\arg z(t)+q(t)$, and the unit disk is mapped to the domain bounded by the given smooth boundary  $z(t)$ with the help of the Cauchy integral formula. So we construct an approximate polynomial conformal mapping. Similar method was also applied for construction of the annulus conformal mapping onto an arbitrary multiconnected domain with the smooth boundary in \cite{ShAb, ShAb2}. 

Note that  we can reconstruct $q'(t)$ intead of $q(t)$ in the case of smooth boundary \cite{jca}. So this method  can also be considered as one of the methods using the derivatives \cite{Port}. %Then $\alpha_0$ turns simply into the integration constant and we consider only the coefficients with  $\cos n t, \sin n t$, $n \geq 1$.

 The drawback of the reparametrization method is that it does not cover the conformal mappings of the unit  disk onto domains with non-smooth boundaries.  For instance, in the case of a domain with the angle $\phi$ for $t=t_0$ equation (\ref{2steqm}) turns into
 \begin{eqnarray}\label{3steqm}
\phi q(t)=  \int\limits_{0}^{2 \pi}q(\tau) \frac{ \partial(\mathrm{arg}(z(\tau)-z(t)))}{\partial \tau} d \tau+\nonumber \\
+ \int\limits_{0}^{2 \pi} \ln |z(\tau)|  \frac{\partial [\ln |z(\tau)-z(t)|]}{\partial \tau}d\tau
\end{eqnarray}
at the point $t_0$. In order to overcome this difficulty we apply the additional conformal mapping which "straightens" the bondary curve at the corresponding point. Then we apply the reparamentrization method to the new domain with the smooth boundary and again apply the conformal mapping that "bends" the boundary  back to the initial one. Our aim is to represent this final mapping as the polynomial fraction.

In the article we apply the modification of the  conformal mapping construction of  \cite{jca} both for domains with boundary   angles and for slender regions. We present the mapping as a polynomial fraction. We first show that the method of \cite{jca} is applicable to domains with acute external angles. Then we present the  polynomial fraction construction for the internal angle equal to $\pi/2$ and conformally map the unit disk to the domain with such an angle. After that we construct the polynomial fraction for the angles $k\pi/N$, $k<N \in \mathbb{N}$. Finally we show that this approach  is valid for the conformal mapping of the unit disk to the slender region.

\section{The case of an internal angle greater than $\pi$}

The  method of  \cite{jca} allows us to solve the conformal mapping construction problem for any contour  with the boundary curve forming internal angles  greater than $\pi$. This can be illustrated by  certain examples.

Let the angle point correspond to the value $ 0 $ of the  parameter, the internal angle be equal to $ \pi \alpha, 2> \alpha> 1$. Then the representation of the boundary curve equation   in the neighborhood of the angle point has the form $ z (t) = (1-e^{i t})^{\alpha} K$,  $K \in \mathbb{R}$.
The difference between the Fourier series  partial sum $S_n (x)$ and the function $ f (x) $ itself is expressesed by the formula $ S_n (x) -f (x) = \frac{1}{\pi} \int \limits_{0 }^{\varepsilon} \frac{1}{2} (f (x + t) + f (x-t) -2f (x)) \frac{\sin (nt)}{t} dt + o (1)$ (\cite{Zig}, Chapter 2, formula (7.1)). In this case, $S_n (0) -z (0) = K\frac{F (n, \varepsilon, \alpha)}{\pi} + o (1) $, where
\begin{equation}\label{FN}
F(n, \varepsilon, \alpha)=2^{\alpha+1}\int\limits_{0}^{\varepsilon}\sin^{\alpha}(\frac{t}{2})  \cos(\frac{\alpha}{2}(t-\pi))\frac{\sin(nt)}{t}dt.
\end{equation}
  For $ \varepsilon \leq \frac{\pi}{2 n} $ we obtain
$$
F(n, \varepsilon, \alpha)=2^{\alpha+1}\int\limits_{0}^{\varepsilon} \sin^{\alpha}(\frac{t}{2}) \cos(\frac{\alpha}{2}(t-\pi))\frac{\sin(nt)}{t}dt\leq
$$
$$ \leq 2 \int\limits_{0}^{\varepsilon} t^{\alpha}\frac{\sin(nt)}{t}dt \leq 2  \int\limits_{0}^{\varepsilon} n t^{\alpha} dt=
$$
$$
=2 \frac{n\varepsilon^{\alpha+1}}{\alpha+1}\leq  \frac{\pi^{\alpha+1}}{2^{\alpha} n^{\alpha}(\alpha+1)}\leq \frac{\pi}{2}(\frac{\pi}{2 n})^{\alpha} \leq \frac{\pi^2}{4 n}.
$$ 
Hence for all $\alpha> 1 $ the difference between the values of the Fourier series partial sum $ S_n (x) $ and the function $ z (t) $ itself can be made arbitrarily small for a sufficiently large $ n$. That is, we have the convergence of the Fourier series at the angle point, regardless of the angle. This allows us to apply the conformal mapping construction method of \cite{jca}.

Example 1. Consider the piecewise circular contour (two semicircles and one circle quarter) with the external angle  $\pi/2$ (Fig.1).
First we approximate the boundary with a Fourier polynomial of degree $10$. Then we construct the approximating polynomial of degree $50$.
 \begin{figure}\label{fig1}
   \begin{center}
\includegraphics[width=4truecm,height=4truecm]{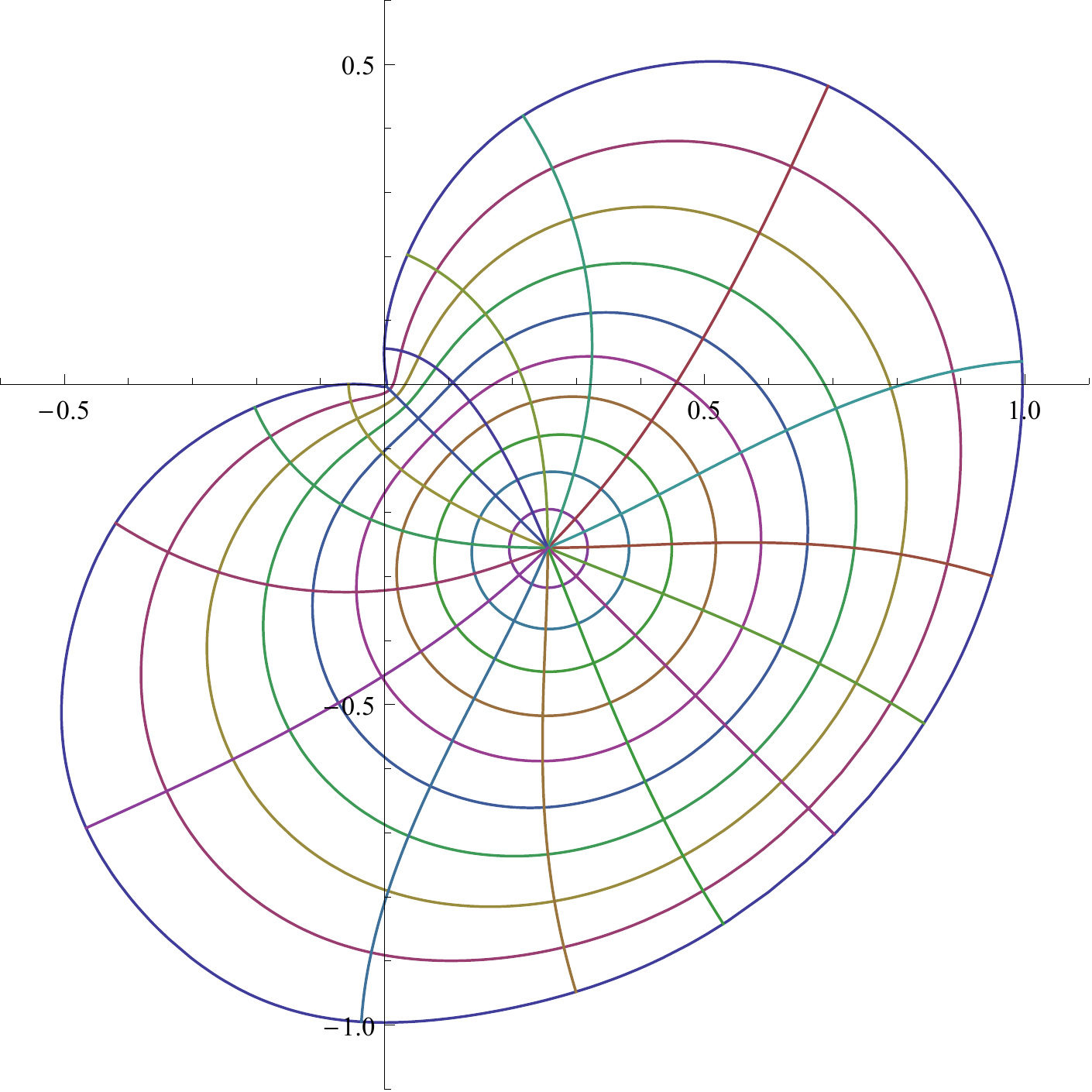}
\caption{The approximation of the  contour with the external angle  $\pi/2$ and the polar net image.}
  \end{center}
\end{figure}

Example 2. The  three-semicircle contour   with the external angle  $0$ (Fig.2). Again we first approximate the boundary with a Fourier polynomial of degree $10$. We then  construct the approximating polynomial of degree $50$.

 \begin{figure}\label{fig2}
   \begin{center}
\includegraphics[width=5truecm,height=4truecm]{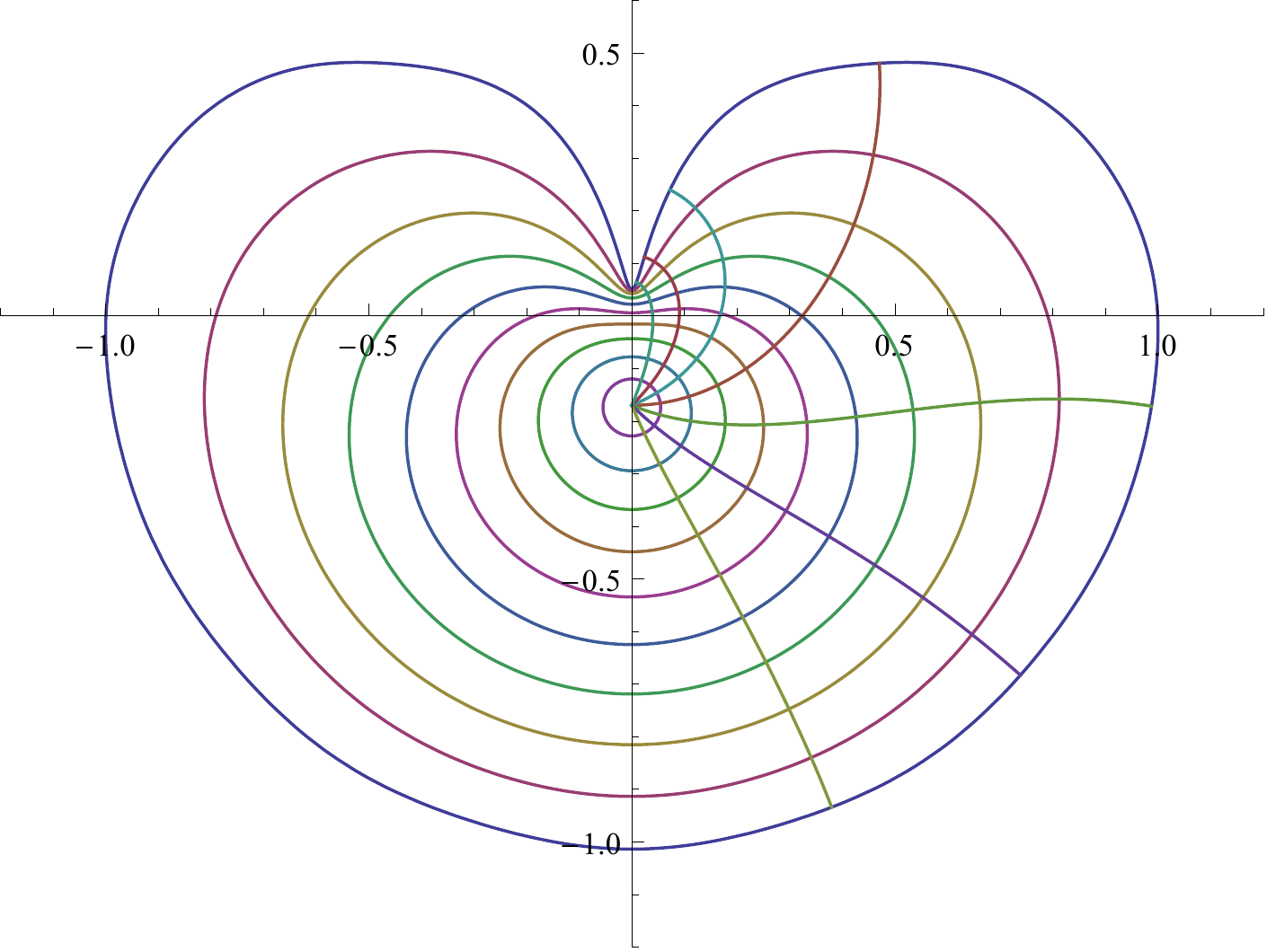}
\caption{The approximation of the  contour with the external angle  $0$ and a part of the polar net image}
  \end{center}
\end{figure}

The similar example for the doubly connected domain with rectangular inner boundary can be found in \cite{ShAb2}.
  
\section{The construction scheme for the case of an internal angle  less than $\pi$}

It is computationally difficult to apply the conformal mapping construction of \cite{jca} for a domain whose boundary forms an  acute internal angle. Then the mapping polynomial converges slowly and the resulting conformal mapping angle point does not look like an angle at all (sort of a bubble). 

 Consider a curve whose behavior at an angle point is similar to    $(1-e^{i t})^{\alpha}K$,  $K \in \mathbb{R}$, with $\alpha \in (0, 1)$ in a neighborhood of $t=0$.
    Then again by Dini criterion  ( \cite{Zig}, Chapter 2, Section 6) we have a singularity of type $t^{\alpha-1}$ at $t=0$ and the Fourier series slowly converges at $t=0$. In order to esimate $F(n, \varepsilon, \alpha)$, defined by relation  (\ref{FN}), we consider the following inequalities: $\sin^{\alpha}(t/2) \geq (\frac{ t}{\pi })^{\alpha}$, $\frac{\sin (nt)}{t} \geq \frac{2 n}{\pi}$, $\cos(\frac{\alpha}{2} (t -\pi)) \geq \cos(\frac{\alpha}{2})$. 
Then we have $F(n, \frac{\pi}{2 n}, \alpha) \geq \cos(\frac{\alpha \pi}{2}) \frac{2}{ n^{\alpha}(\alpha+1)}$ for  $\varepsilon =\frac{\pi}{2 n}$. Hence $F(n, \frac{\pi}{2 n}, \frac{1}{\ln n}) \to \frac{2}{e}$  as $n \to \infty$. So, for the singular point $ t = 0$, the Fourier series  rate of  convergence  to the generating function $ z (t) $ is the less the closer $ \alpha $ is to $ 0$. Thus, the method from \cite{jca} is difficult to apply, since even the Fourier series poorly approximate a curve with such an angle point.

Let the domain boundary be angled and the angle be equal to $k\pi/N$, $N =2,3, \ldots$, $k \in\{1, 2, \ldots,  N-1\}$.
 
 The main idea of the mapping construction is to first put the angle point at $0$, make the domain smooth with the mapping  $z^{N/k}$, construct the conformal mapping onto this smooth domain and then apply the fraction polynomial approximation of the inverse mapping $\sqrt[N]{z^k}$. Note that the domain should completely lie in the right half-plane. In the other case we should apply fraction linear mapping in order to put the domain into the angle.
  
The continued fraction converges to $\sqrt{x}$ at $x=0$ faster than the Taylor expansion of the function $\sqrt{\frac{x-a}{a}+1}$ into degrees of $(x-a)$ to the function $\sqrt{\frac{x-a}{a}+1}$ itself  at $x=0$ \cite{sf}. The most thorough and refined method here is the  Pad\'e rational function approximation of the algebraic function \cite{apt, apt2}. Note that these approximations are optimal in the set of fraction polynomials though their construction requires application of Euclidean algorythm  and additional investigation of the holomorphness domain $D$.

The main result here is that the recursively constructed relations converge to  the continued fraction approximating any rational root $\sqrt[N]{z}$, $N \in \mathbb{N}$. The constructed sequence is clearly not  Pad\'e one. But the construction itself is fairly simple, does not possess nonunique solutions and  provides convergence to the root function at the complex right half-plane. Similar results can be found in \cite{sf}. Also the author is sure that this result can be proved along the lines of \cite{Kh}. Again the proof should apply induction and we need to consider the roots of the polynomials instead of the mapping itself. Note also that the fractional polynomial mappings can be applied, for instance, to exact solution of the elasticity theory problems \cite{book}.

\section{The square root approximation}

First consider the basic problem of the square root fraction polynomial representation. It is well-known that $\sqrt{z}=1+\frac{z-1}{\sqrt{z}+1}$. This gives rise to the following recursive procedure:

\begin{lemma} Assume that $f_n(z)=1+\frac{z-1}{1+f_{n-1}(z)}$, $f_1(z)=1+\frac{z-1}{1+z}$. Then the following facts hold true for $z$ with $\mathrm{Re}[z]>0$:

1. $\mathrm{Re}[f_{n}(z)]>0$

2. $\mathrm{Im}[f_{n}(z)]$ has the same sign as $Im[z]$.

3. The fraction $\frac{\mathrm{Im}[f_{n}(z)]}{\mathrm{Re}[f_{n}(z)]}$ has the same sign as the fraction $\frac{\mathrm{Im}[z]}{\mathrm{Re}[z]}$ and $|\frac{\mathrm{Im}[f_{n}(z)]}{\mathrm{Re}[f_{n}(z)]}|<|\frac{\mathrm{Im}[z]}{\mathrm{Re}[z]}|$.
\end{lemma}

\begin{proof}

The proof is by induction on $n$.

The induction base is $f_1(z)=1+\frac{z-1}{1+z}$.

1. $\mathrm{Re}[f_1(z)]=\frac{2 |z|^2+2 \mathrm{Re}[z]}{|z+1|^2}>0$.

2. $\mathrm{Im}[f_1(z)]=\frac{2 \mathrm{Im}[z]}{|z+1|^2}$ is of the same sign as $\mathrm{Im}[z]$.

3. Assume that $\mathrm{Im}[z]>0$, then $\frac{\mathrm{Im}[f_{n}(z)]}{\mathrm{Re}[f_{n}(z)]}=\frac{2 \mathrm{Im}[z]}{2 \mathrm{Re}[z]+2 |z|^2}<\frac{\mathrm{Im}[z]}{\mathrm{Re}[z]}$.

The induction step then is as follows:

1. The nominator similar to that of the induction base is $\mathrm{Re}[z]+\mathrm{Re}[z] \mathrm{Re}[f_{n-1}(z)]+\mathrm{Im}[z]\mathrm{Im}[f_{n-1}(z)]+\mathrm{Re}[f_{n-1}(z)]+|f_{n-1}(z)|^2>0$ by conjecture.

2. Similarly the sign of $\mathrm{Im}[f_{n}(z)]$ coincides with the sign  of $\mathrm{Im}[f_{n-1}]+\mathrm{Im}[z]+$ 

$+\mathrm{Im}[z]\mathrm{Re}[f_{n-1}(z)]-\mathrm{Re}[z]\mathrm{Im}[f_{n-1}(z)]$. The last two summands are of the same sign as $\mathrm{Im}[z]$ by conjecture.

3. Consider $\frac{\mathrm{Im}[f_{n-1}(z)]+\mathrm{Im}[z]+\mathrm{Im}[z]\mathrm{Re}[f_{n-1}(z)]-\mathrm{Re}[z] \mathrm{Im}[f_{n-1}(z)]}{\mathrm{Re}[f_{n-1}(z)]+\mathrm{Re}[z]+\mathrm{Re}[z] \mathrm{Re}[f_{n-1}(z)]+\mathrm{Im}[z]\mathrm{Im}[f_{n-1}(z)]+|f_{n-1}(z)|^2}$. Note that the respective summands of the nominator and denominator meet the desired relation so the fraction itself is less in modulus than $\frac{\mathrm{Im}[z]}{\mathrm{Re}[z]}$.
\end{proof}

\begin{statement} There are no points at the right complex half-plane at which the derivative of $f_{n}(z)$ vanishes.
\end{statement}
\begin{proof}
1. Consider $z$ so that $\mathrm{Im}[z] \neq 0$. Then by item 3 of Lemma 1 we have  $\frac{\partial f_n(z)}{\partial\arg(z)} \neq 0$. 
Indeed for $n=1$, locally $\arg(f_1(z))=\arg(z)-\arg(z+1) = k_1\arg(z)$, $0<k_1<1$ by item 3 of Lemma 1. Let $\mathrm{Im}[z]>0$. Then $\forall n$ and $\arg(f_n(z))=\arg(f_{n-1}(z)+z)-\arg(f_{n-1}(z)+1)$ we have $k_{n-1}\arg(z)<\arg(f_{n-1}(z)+z)<\arg(z)$ and $\arg(f_{n-1}(z)+1)= \tilde{k}_{n-1}\arg(z) < k_{n-1}\arg(z)$. So $0<(k_{n-1}-\tilde{k}_{n-1})\arg(z)<\arg(f_n(z))<(1- \tilde{k}_{n-1})\arg(z)$.

2. Consider $z=x \in \mathbb{R}^+$. Then we must prove that $f_{n-1}'(x)+1+f_{n-1}(x)-xf_{n-1}'(x)>0$ or more precisely, that $f_{n-1}(x)-xf_{n-1}'(x)>0$. The proof is by induction. Base $f_1'(x)=\frac{2}{(x+1)^2}$. The induction step is as follows: Consider $f_{n}(x)-xf_{n}'(x)=\frac{f_{n-1}(x)+x}{1+f_{n-1}(x)}-x\frac{f_{n-1}'(x)+1+f_{n-1}(x)-xf_{n-1}'(x)}{(1+f_{n-1}(x))^2}=\frac{f_{n-1}(x)+f_{n-1}^2(x)-xf_{n-1}'(x)(1-x)}{1+f_{n-1}(x)}>0$ by the induction assumption and item 1 of Lemma 1.

Also since $\mathrm{Re}[1+f_{n-1}(z)]>0$ for $z$ from the right half plane the function $f_{n}(z)$ does not have poles in this set.
\end{proof}

\begin{theorem} The functions $f_n(z)=1+\frac{z-1}{1+f_{n-1}(z)}$ converge to $\sqrt{z}$  with the convergence rate $(\frac{1-\sqrt{z}}{1+\sqrt{z}})^n$ for $z$, $\mathrm{Re}[z]>0$.
\end{theorem}

\begin{proof}
First note that  $|\frac{1-\sqrt{z}}{1+\sqrt{z}}|<1$ for $z$, $\mathrm{Re}[z]>0$.

Consider  $\varepsilon=z-\sqrt{z}$. Then $\displaystyle f_1(z)=\sqrt{z}+\frac{\varepsilon(1-\sqrt{z})}{1+\sqrt{z}+\varepsilon}$, $\displaystyle f_2(z)=\sqrt{z}+\frac{\varepsilon(1-\sqrt{z})^2}{(1+\sqrt{z})^2+\varepsilon((1+\sqrt{z})+(1-\sqrt{z}))}$, $\ldots$,
$\displaystyle f_n(z)=\sqrt{z}+\frac{\varepsilon(1-\sqrt{z})^n}{(1+\sqrt{z})^n+\varepsilon K_n}$, here $\displaystyle K_n=(1+\sqrt{z})^{n-1}+K_{n-1}(1-\sqrt{z})$, $K_1=1$. 

Hence $\displaystyle K_n=(1+\sqrt{z})^{n-1}(1+(\frac{1-\sqrt{z}}{1+\sqrt{z}})\frac{K_{n-1}}{(1+\sqrt{z})^{n-2}})=\ldots=(1+\sqrt{z})^{n-1}(1+\frac{1-\sqrt{z}}{1+\sqrt{z}}+\ldots+(\frac{1-\sqrt{z}}{1+\sqrt{z}})^{n-1})$. Thus $\displaystyle \lim\limits_{n\to \infty}\frac{K_n}{(1+\sqrt{z})^{n-1}}=\frac{1}{1-\frac{1-\sqrt{z}}{1+\sqrt{z}}}=\frac{1+\sqrt{z}}{2\sqrt{z}}$.

So, $$
\lim\limits_{n\to \infty}\frac{\varepsilon(1-\sqrt{z})^n}{(1+\sqrt{z})^n+\varepsilon K_n}=\lim\limits_{n\to \infty}\frac{\varepsilon(1-\sqrt{z})^n}{(1+\sqrt{z})^n (1+\varepsilon \frac{K_n}{(1+\sqrt{z})^{n}})}=
$$
$$
\displaystyle =\varepsilon \lim\limits_{n\to \infty}\frac{(1-\sqrt{z})^n}{(1+\sqrt{z})^n (1+\frac{\varepsilon}{1+\sqrt{z}} \frac{K_n}{(1+\sqrt{z})^{n-1}})}=(z-\sqrt{z})\lim\limits_{n\to \infty}\frac{(1-\sqrt{z})^n}{(1+\sqrt{z})^n (1+\frac{\sqrt{z}-1}{2})}=
$$
$$
=\frac{z-\sqrt{z}}{ 1+\frac{\sqrt{z}-1}{2}}\lim\limits_{n\to \infty}\frac{(1-\sqrt{z})^n}{(1+\sqrt{z})^n }=0.
$$
This completes the proof.
\end{proof}

Assume now that we have a convex domain with acute internal angles and we need to construct the conformal mapping of the unit disk onto this domain.
The main construction steps are as follows: we make the domain as round as possible with square mappings. If the resulting domain does not overlap itself then we construct the approximating polynomial according to the method of \cite{jca}. Finally we construct the square root approximations  of the resulting image inverse to the squares of the first step.

Example 3. Let us construct an approximate conformal map of the unit disk onto the contour with the internal angle  $\pi/2$. Here we have the 11th iteration of the square root approximation and degree 50 polynomial for the initial domain (Fig.3).

  \begin{figure}\label{fig3}
   \begin{center}
\includegraphics[width=4truecm,height=4truecm]{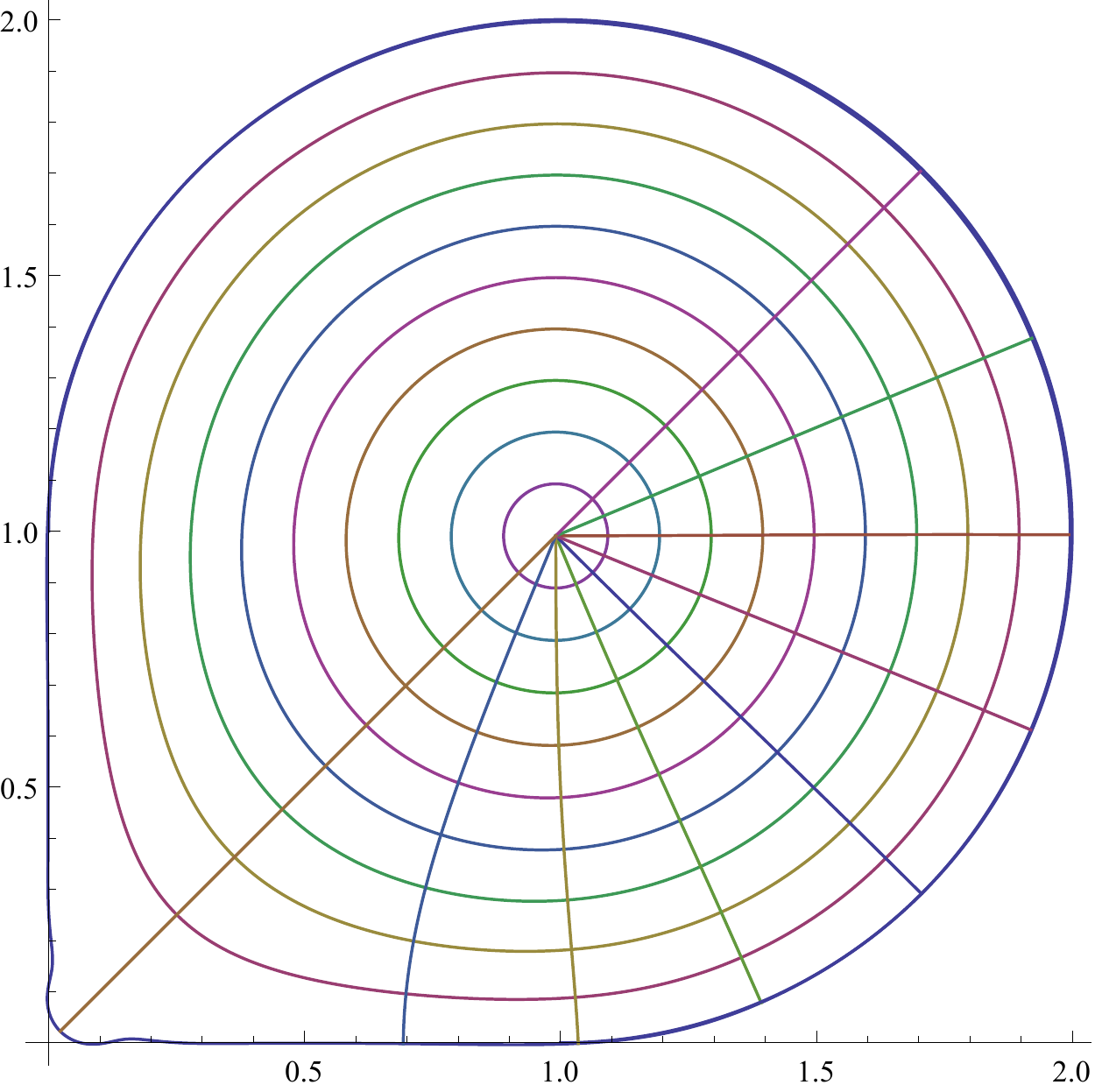}
\caption{Contour with the internal angle equal to $\pi/2$}
  \end{center}
\end{figure}

\section{The case of $z^{\frac{k}{N}}$}
Consider a natural number $N \in \mathbb{N}$.
The recursive representation of the $N$-th root then takes the following form:
$$
\displaystyle z^{\frac{1}{N}}=1+\frac{z-1}{z^{\frac{N-1}{N}}+\ldots+z^{\frac{2}{N}}+z^{\frac{1}{N}}+1}=
$$
$$
\displaystyle =1+\frac{z-1}{\frac{z}{z^{\frac{1}{N}}}+\frac{z}{z^{\frac{2}{N}}}+\ldots+z^{\frac{2}{N}}+z^{\frac{1}{N}}+1}
$$
Also at the same time we have
$$
\displaystyle z^{\frac{k}{N}}=z^{\frac{k-1}{N}}+\frac{z-z^{\frac{k-1}{N}}}{z^{\frac{N-k}{N}}+\ldots+z^{\frac{2}{N}}+z^{\frac{1}{N}}+1}
$$
Then we express the right-hand side of the first relation only through $z^{\frac{1}{N}}$.
Certain statements similar to that of Section 4 hold and we have the convergence of these fractions to the relative function degrees at the complex right half-plane. The only principally new appearance here is  simultaneous proof of the result for all  the representations of $z^{\frac{k}{N}}$, $k =1, 2 \ldots, N-1$. Here we need the additional induction step for the recursive representations of $z^{\frac{k}{N}}$ and $\frac{z}{z^{\frac{k}{N}}}$, $k =1, 2 \ldots, [N/2]$.
Again the most important part is  the first construction step. Assume then that for any $k \in \{1, 2 \ldots, [N/2]\}$ the first approximation of  $z^{\frac{k}{N}}$ equals $1+\frac{z-1}{z+1}$.

\begin{theorem}
For any $z$, $\mathrm{Re}[z]>0$, $N =2, 3, \ldots$, and  $k \in \{1, \ldots, N-1\}$ the sequence of fraction-polynomial approximations converges to $\displaystyle z^{\frac{k}{N}}$ with the convergence rate 
$$\displaystyle
|z^{(k-1)/N} \left(\frac{z-N(z^{1/N}-1)z^{\frac{[N/2]}{N}}-1}{z-1}\right)^n|.
$$
\end{theorem}

%Consider the case of $z$, $|z|\leq 1$.
\begin{proof}
Let the first approximation of $z^{1/N}$ equal $z^{1/N}+ \varepsilon$. Then the first approximation of $z^{k/N}$ equals $z^{k/N}+ k z^{1/N}\varepsilon+o(\varepsilon)$ and that of 
$z^{\frac{N-k}{N}}$ is $z^{\frac{N-k}{N}}-k \varepsilon z^{\frac{N-k-1}{N}} +o(\varepsilon)$, $k=1, \ldots, [\frac{N}{2}]$. We have the next approximation of $z^{1/N}$ equal to
$$
f_2(z)=1+(z-1)/[1+z^{1/N}+\ldots+z^{\frac{N-1}{N}}+
$$
$$
+\varepsilon(1+2z^{1/N}+\ldots+[N/2]z^{\frac{[N/2]-1}{N}}
-([N/2]-1)z^{\frac{[N/2]}{N}}-\ldots -z^{\frac{N-2}{N}})]=
$$
$$
=z^{1/N}+\{\varepsilon(1+2z^{1/N}+\ldots+[N/2]z^{\frac{[N/2]-1}{N}}-
$$
$$
-([N/2]-1)z^{\frac{[N/2]}{N}}-\ldots -z^{\frac{N-2}{N}})(1-z^{1/N})\}/\{1+z^{1/N}+\ldots+z^{\frac{N-1}{N}}+
$$
$$
+\varepsilon(1+2z^{1/N}+\ldots+[N/2]z^{\frac{[N/2]-1}{N}}-([N/2]-1)z^{\frac{[N/2]}{N}}-\ldots -z^{\frac{N-2}{N}})\}=
$$
$$
=z^{1/N}+\varepsilon\{z^{\frac{N-1}{N}}+\ldots+z^{\frac{[N/2]+1}{N}}-(N-1)z^{\frac{[N/2]}{N}}+z^{\frac{[N/2]-1}{N}}+\ldots+1\}/\{1+z^{1/N}+\ldots+z^{\frac{N-1}{N}}+
$$
$$
+\varepsilon(1+2z^{1/N}+\ldots+[N/2]z^{\frac{[N/2]-1}{N}}-([N/2]-1)z^{\frac{[N/2]}{N}}-\ldots -z^{\frac{N-2}{N}})\}.
$$
The third approximation then equals
$$
f_3(z)=z^{1/N}+\varepsilon(z^{\frac{N-1}{N}}+\ldots+z^{\frac{[N/2]+1}{N}}-(N-1)z^{\frac{[N/2]}{N}}+z^{\frac{[N/2]-1}{N}}+\ldots+1)^2/[(1+z^{1/N}+\ldots+z^{\frac{N-1}{N}})^2+
$$
$$
+\varepsilon(1+2z^{1/N}+\ldots+[N/2]z^{\frac{[N/2]-1}{N}}-([N/2]-1)z^{\frac{[N/2]}{N}}-\ldots -z^{\frac{N-2}{N}})((1+z^{1/N}+\ldots+z^{\frac{N-1}{N}})+
$$
$$
+(z^{\frac{N-1}{N}}+\ldots+z^{\frac{[N/2]+2}{N}}-(N-1)z^{\frac{[N/2]+1}{N}}+z^{\frac{[N/2]}{N}}+\ldots+1))].
$$
Again as in the square root case we have 
$$
f_n(z)=z^{1/N}+\varepsilon\frac{(z^{\frac{N-1}{N}}+\ldots+z^{\frac{[N/2]+1}{N}}-(N-1)z^{\frac{[N/2]}{N}}+z^{\frac{[N/2]-1}{N}}+\ldots+1)^n}{(1+z^{1/N}+\ldots+z^{\frac{N-1}{N}})^n+\varepsilon K_n}.
$$
Here $K_n=K_{n-1} ((1+z^{1/N}+\ldots+z^{\frac{N-1}{N}})+(z^{\frac{N-1}{N}}+\ldots+z^{\frac{[N/2]+1}{N}}-(N-1)z^{\frac{[N/2]}{N}}+z^{\frac{[N/2]-1}{N}}+\ldots+1))$. 

Since $\displaystyle |\frac{K_n}{(1+z^{1/N}+\ldots+z^{\frac{N-1}{N}})^n}|$ is bounded by 
$$
\displaystyle |1+\frac{z^{\frac{N-1}{N}}+\ldots+z^{\frac{[N/2]+1}{N}}-(N-1)z^{\frac{[N/2]}{N}}+z^{\frac{[N/2]-1}{N}}+\ldots+1}{1+z^{1/N}+\ldots+z^{\frac{N-1}{N}}}|
$$ 
and  
$$
\displaystyle \frac{|z^{\frac{N-1}{N}}+\ldots+z^{\frac{[N/2]+1}{N}}-(N-1)z^{\frac{[N/2]}{N}}+z^{\frac{[N/2]-1}{N}}+\ldots+1|}{|1+z^{1/N}+\ldots+z^{\frac{N-1}{N}}|}<1
$$ 
we have the convergent sequence for any $z$, $\mathrm{Re}[z]>0$.

\begin{lemma}
For $z$ such that $\mathrm{Re}[z]>0$ and any $N =2, 3, \ldots$, we have
$$
\displaystyle \frac{|z^{\frac{N-1}{N}}+\ldots+z^{\frac{[N/2]+1}{N}}-(N-1)z^{\frac{[N/2]}{N}}+z^{\frac{[N/2]-1}{N}}+\ldots+1|}{|1+z^{1/N}+\ldots+z^{\frac{N-1}{N}}|}<1.
$$
\end{lemma}

\begin{proof}
In order to prove the relation consider $z$ such that $\mathrm{Re}[z]>0$, $\mathrm{Im}[z]\geq 0$. Consider $t=z^{1/N}=r^{1/N} e^{i\phi}$, $\phi\in (-\frac{\pi}{2N}, \frac{\pi}{2N})$. The maximal value of the fraction happens for $\phi=\pm \frac{\pi}{2N}$ as a boundary value of an analytic function on $t$. Indeed then we have the function $\displaystyle \frac{t^{N-1}+\ldots+t^{[N/2]+1}-(N-1)t^{[N/2]}+t^{[N/2]-1}+\ldots+1}{1+t+\ldots+t^{N-1}}$ that does not possess poles in the right half-plane since the denominator real part is strictly positive for $t$ such that $\mathrm{Re}[t]>0$, $|arg[t]| <\frac{\pi}{2N}$.

Assume first that $N$ is even.
 Consider the real part of the nominator
$$
\displaystyle r^{\frac{N-1}{N}}\cos(\frac{(N-1)\pi}{2 N})+\ldots +r^{\frac{1}{N}}\cos(\frac{\pi}{2 N})+1-N r^{1/2}\cos(\pi/4)=
$$
$$\displaystyle
=r^{\frac{N-1}{N}}\sin(\frac{\pi}{2 N})+\ldots+r^{\frac{1}{2}-\frac{1}{N}}\sin(\frac{(N-2)\pi}{2N})+ r^{1/2}\cos(\pi/4)+
$$
$$
 + r^{\frac{1}{2}+\frac{1}{N}}\cos(\frac{(N-2)\pi}{2N})+\ldots +r^{\frac{1}{N}}\cos(\frac{\pi}{2 N})+1-N r^{1/2}\cos(\pi/4).
$$
Similarly for the imaginary part we have
$$\displaystyle
r^{\frac{N-1}{N}}\cos(\frac{\pi}{2 N})+\ldots+r^{\frac{1}{2}-\frac{1}{N}}\cos(\frac{(N-2)\pi}{2N})+ r^{1/2}\sin(\pi/4)+
$$
$$
 + r^{\frac{1}{2}+\frac{1}{N}}\sin(\frac{(N-2)\pi}{2N})+\ldots +r^{\frac{1}{N}}\sin(\frac{\pi}{2 N})-N r^{1/2}\sin(\pi/4).
$$
In order to compare the absolute values of the nominator and denominator we project their components onto the same line $\phi=\pi/4$ since the absolute value of the negative deformation $-N r^{1/2}(\cos(\pi/4)+\sin(\pi/4))$ is maximal in this direction.
For any $k=1, \ldots, N-1$, the absolute value of $z^{\frac{k}{N}}$ projection onto this line  is  $\sqrt{2}$ times less than the number  $z^{\frac{k}{N}}$ real and imaginary part sum.

Then for any $k=1, \ldots, [N/2]-1$, we have
$$\displaystyle
r^{1-\frac{k}{N}}\sin(\frac{k\pi}{2N})+r^{\frac{k}{N}}\cos(\frac{k\pi}{2N})+r^{1-\frac{k}{N}}\cos(\frac{k\pi}{2N})+r^{\frac{k}{N}}\sin(\frac{k\pi}{2N})-\sqrt{2}r^{1/2}=
$$
$$\displaystyle
=r^{1/2}(r^{1/2-\frac{k}{N}}\sin(\frac{k\pi}{2N})+r^{\frac{k}{N}-1/2}\cos(\frac{k\pi}{2N})+r^{1/2-\frac{k}{N}}\cos(\frac{k\pi}{2N})+r^{\frac{k}{N}-1/2}\sin(\frac{k\pi}{2N})-\sqrt{2})\geq
$$
$$\displaystyle
\geq r^{1/2}(r^{\frac{k}{N}-1/2}\cos(\frac{k\pi}{2N})+r^{1/2-\frac{k}{N}}\cos(\frac{k\pi}{2N})-\sqrt{2})\geq
$$
$$\displaystyle
\geq r^{1/2}\sqrt{2}(r^{\frac{k}{N}-1/2}+r^{1/2-\frac{k}{N}}-2)
\geq 0.
$$

The equality happens only for $z=0$.

Let $N$ be an odd number. Then similarly to the even case we project our sums onto the line $\phi=\frac{[N/2]\pi}{2N}$ and gather $z^{\frac{k}{N}}$ and $z^{\frac{N-1-k}{N}}$, $k=0, \ldots, [N/2]$. The projection absolute value for any $k=0, \ldots, N-1$, equals $\mathrm{Re}[z^{\frac{k}{N}}]\cos(\frac{[N/2]\pi}{2N})+\mathrm{Im}[z^{\frac{k}{N}}]\sin(\frac{[N/2]\pi}{2N})$.

The relation we need then equals
$$\displaystyle
(r^{1-\frac{k+1}{N}}\cos(\frac{(N-k-1)\pi}{2N})+r^{\frac{k}{N}}\cos(\frac{k\pi}{2N})) \cos(\frac{[N/2]\pi}{2N})+
$$
$$\displaystyle
+(r^{1-\frac{k+1}{N}}\sin(\frac{(N-k-1)\pi}{2N})+r^{\frac{k}{N}}\sin(\frac{k\pi}{2N})) \sin(\frac{[N/2]\pi}{2N})-
$$
$$\displaystyle
-r^{\frac{[N/2]}{N}}(\cos^2(\frac{[N/2]\pi}{2N})+\sin^2(\frac{[N/2]\pi}{2N}))=
$$
$$\displaystyle
=r^{1-\frac{k+1}{N}}\cos(\frac{([N/2]-k)\pi}{2N})+r^{\frac{k}{N}}\cos(\frac{(k-[N/2])\pi}{2N})-r^{\frac{[N/2]}{N}}\geq 
$$
$$
\geq r^{\frac{[N/2]}{N}} ((r^{\frac{N-1}{2N}-\frac{k}{N}}+r^{\frac{k}{N}-\frac{N-1}{2 N}})\frac{1}{\sqrt{2}}-1) \geq
0.
$$
The lemma is proved.
\end{proof}

The convergence rate for $z^{k/N}$ can be estimated by some multiple of 
$$\displaystyle |z^{(k-1)/N} \left(\frac{z^{\frac{N-1}{N}}+\ldots+z^{\frac{[N/2]+1}{N}}-(N-1)z^{\frac{[N/2]}{N}}+z^{\frac{[N/2]-1}{N}}+\ldots+1}{1+z^{1/N}+\ldots+z^{\frac{N-1}{N}}}\right)^n|=
$$
$$
=|z^{(k-1)/N} \left(\frac{z-N(z^{1/N}-1)z^{\frac{[N/2]}{N}}-1}{z-1}\right)^n|.
$$ 
So the more acute the angle and the closer $z$ is to $0$ the worse is the approximation convergence. 

This completes the proof of the theorem.
\end{proof}

We now construct the following mappings exactly as in Example 3.

Example 4. The contour (two lines and the circular sector) with the angle $\pi/3$. We first approximate the unfolded domain by the method of \cite{jca} and then fold the result with the fractional polynomial mapping. We apply the recursive formula 
$$
g_n(z)=1+\frac{z-1}{\frac{z}{g_{n-1}(z)}+g_{n-1}(z)+1}, \; g_1(z)=1+\frac{z-1}{z+1}.
$$

The unfolded domain was approximated by the polynomial of degree $50$. We next apply the 6th fraction iteration to fold the domain back to the angled one (Fig.4).

  \begin{figure}[h]\label{fig5}
   \begin{center}
\includegraphics[width=5truecm,height=3truecm]{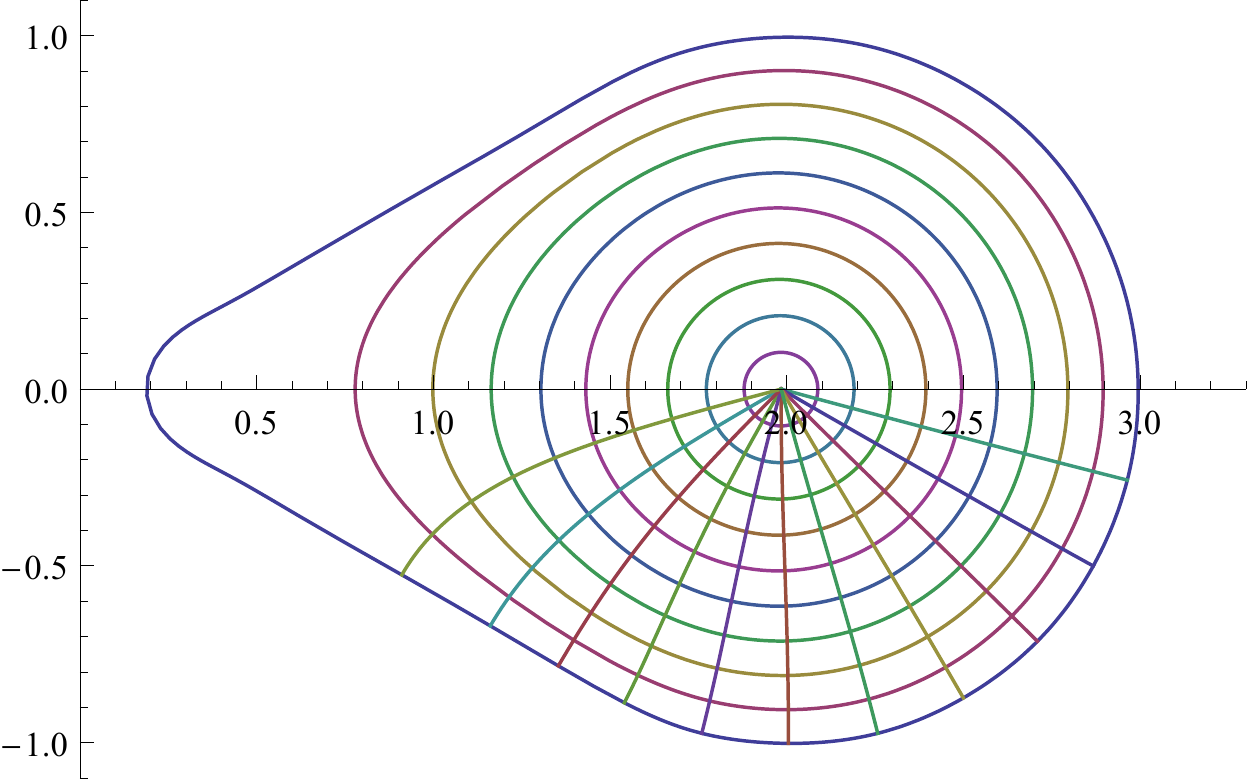}
\caption{Sixth approximation of the contour with the internal angle equal to $\pi/3$ and a part of the polar net image}
  \end{center}
\end{figure}

Example 5. Similar contour (two lines and the circular sector) with the angle $2 \pi/3$. We again approximate the unfolded domain by the method of \cite{jca} and then fold the result with the fractional polynomial mapping. Here the main formula is simply 
$$
h_n(z)=\frac{z}{g_{n}(z)}
$$
for $g_{n}(z)$ of Example 4.
The unfolded domain was approximated by the polynomial of degree $50$. We next apply the 4th fraction iteration to fold the domain back to the angled one (Fig.5). 
  \begin{figure}[h]\label{fig6}
   \begin{center}
\includegraphics[width=4truecm,height=4truecm]{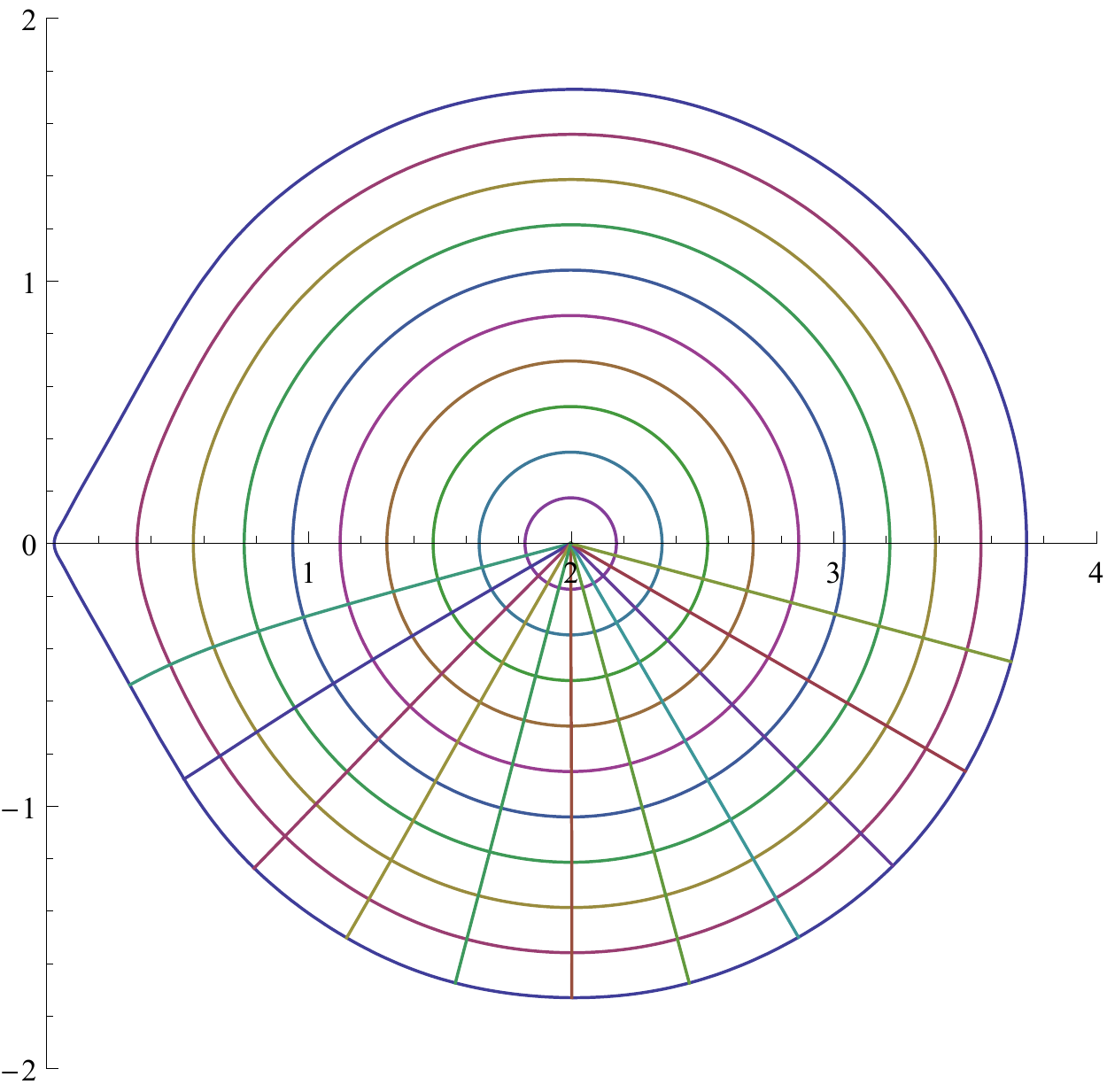}
\caption{Fourth approximation of the contour with the internal angle equal to $2\pi/3$ and a part of the polar net image}
  \end{center}
\end{figure}

These examples show us that the more acute is the internal angle the harder it is to approximate it.

\section{The case of thin domains}

Consider the case of slender regions.
 The second  problem for us is the case of relatively thin domains (e.g. ellipse with two significantly different axes).
Consider  the  integral equation of \cite{jca} kernel behavior for $ \tau $ close to the point $t$ of the largest possible curvature $\kappa(t)$: $ \frac{d}{d \tau} \mathrm{arctg} (\frac{y (\tau) -y (t)} {x (\tau) -x (t)}) = \kappa (t) | z (t) | '/ 2 + o (1)$. Then the diagonal elements of the relative linear equation system  matrix are close to $\kappa(t)$ and are also large. Thus, the greater the curvature $ \kappa (t) $ of the curve in $ t$, the worse the convergence of the polynomial solution.

 The authors of  \cite{UE} numerically solve the singular integral equation in order to find the conformal mappings from elliptic to slender regions.
The method of recursive fractions is also applicable to the conformal mapping construction of a disk onto a thin domain. The main problem here is the so-called point crowding phenomenon.  
Here we achieve the similar results (domain sides ratio $1/4$) with our method as a natural application. We first make the domain less slender with the help of the square mapping $(z-a)^2$, here the point $a$ lies outside the domain and close to its boundary point of maximal curvature. We cannot take this point at the boundary itself since then we achieve the domain that cannot be immediately inserted into the right half-plane at the neighbourhood of $a$. Secondly we apply the approximate conformal mapping construction algorithm.
Finally we apply the square root approximation in order to return to the domain with the given boundary.

 Now, if a domain lies between two sides of the right angle closely to the vertex then we consider  the mapping of the disk onto the squared domain and the square root approximation of  the angle.

Example 6. Consider the ellipse of semiaxes  $1$ and $1/4$: $x^2+16 y^2=1$. Let us construct an approximate conformal mapping of the unit disk onto this ellipse. 

The initial method of \cite{jca} provides us with the following result for the polynomial of degree $1200$ (Fig.6).
  \begin{figure}[h]\label{fig45}
   \begin{center}
\includegraphics[width=8truecm,height=2truecm]{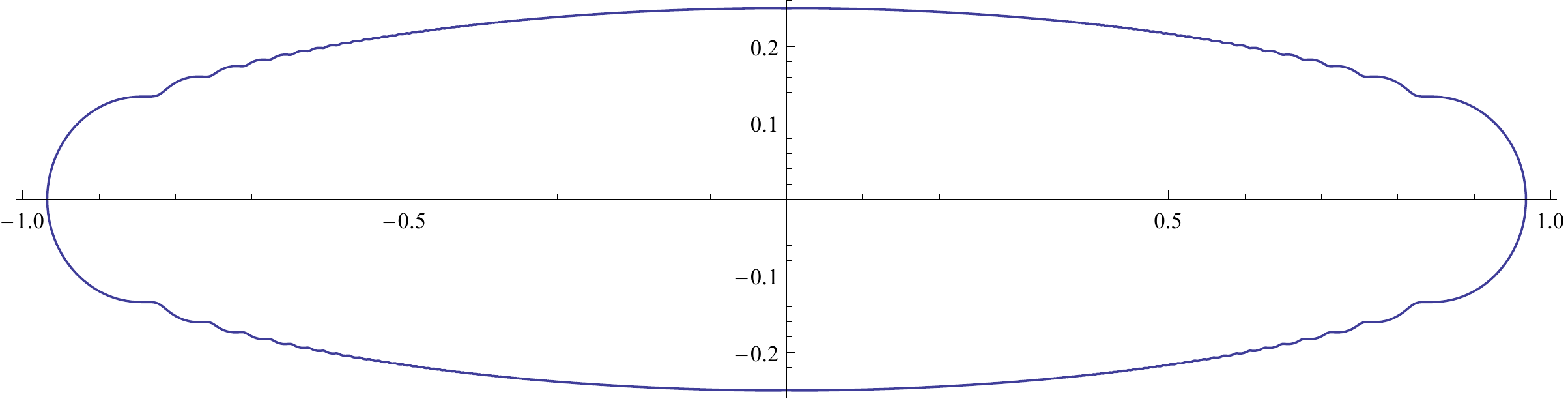}
\caption{The ellipse polynomial approximation}
  \end{center}
    
\end{figure}

  Here we consider the 20th square root iterations and 1000 degree polynomial (Fig.7). Similar picture under only polynomial approximation due to the point crowding phenomenon happens for polynomial of degree $10^4$.

  \begin{figure}[h]\label{fig4}
   \begin{center}
\includegraphics[width=8truecm,height=2.3truecm]{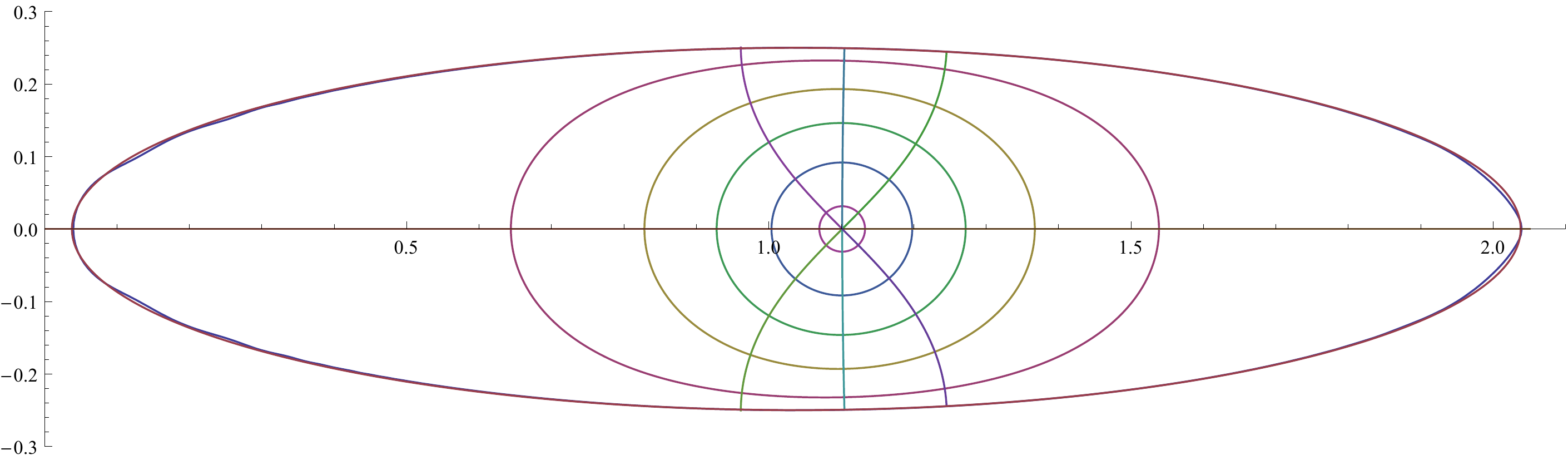}
\caption{The ellipse fraction polynomial approximation}
  \end{center}
\end{figure}

%Examples of Fortran program code and Mathematica file can be downloaded from
%\begin{verbatim}
%https://yadi.sk/d/hwqs_NMb3Xm8XN
%https://yadi.sk/d/usnHZpK93Xm8ZJ
%\end{verbatim}

Pyotr N. Ivanshin

Department of Mechanics and Mathematics, 

Kazan Federal University, 

420008, Kremlyovskaya, 18, 

Kazan, Russia

e-mail:pivanshi@yandex.ru

\begin{thebibliography}{99}

\bibitem{ShAb}  D. F. Abzalilov, E.A.Shirokova, The approximate conformal mapping onto simply and doubly connected domains // Complex Variables and Elliptic Equations - 2016 - P. 1-12.

\bibitem{Ahl}
L. Ahlfors, Complex Analysis: An introduction to the theory of analytic functions of one
complex variable, Third edition, International Series in Pure and Applied Mathematics,
McGraw-Hill Book Co., New York 1978.

\bibitem{apt} A. I. Aptekarev, M. L. Yattselev,
Approximations of algebraic functions by rational ones --- functional analogues of diophantine approximants.  Preprint of Keldysh Institute of Applied Mathematics RAS, Moscow, 2016.

\bibitem{apt2} A. I. Aptekarev, M. L. Yattselev, Pade approximants for functions with branch points -- strong asymptotics of Nuttall--Stahl polynomials, Acta Math., 215 (2015), 217–280 

\bibitem{sf} A. Cuyt, V. B. Petersen,
Brigitte Verdonk, Haakon Waadeland,
William B. Jones, Handbook of Continued Fractions for Special Functions, Springer, 2008.

\bibitem{UE} 
Th. K. DeLillo, A. R. Elcrat,  A Fornberg-like conformal mapping method for slender regions, Journal of Computational and Applied Mathematics, 46 (1993), 49-64.

\bibitem{Dr1} T. A. Driscoll, L. N. Trefethen,
Schwarz-Christoffel mapping, Cambridge Monographs
on Applied and Computational Mathematics, Cambridge University Press, Cambridge
2002.

\bibitem{Dr2} T. A. Driscoll, S. A. Vavasis, Numerical conformal mapping using cross-ratios and
Delaunay triangulation,
SIAM J. Sci. Comput.
19
(1998) 1783--1803

\bibitem{Forn} B. Fornberg, A numerical method for conformal mappings// SIAM J. Sci. Comput., 1980, V.1, N.3, pp.386-400.


\bibitem{Forn1} B. A. Fornberg, A numerical method for conformal mapping of doubly connected regions,
SIAM J. Sci. Statist. Comput. 5 (1984) 771–783.

\bibitem{HA}J. Heinhold, R. Albrecht, Zur Praxis der konformen Abbildung, Rend. Circ. Mat. Palermo
3 (1954) 130–148.

\bibitem{Kh} A.N. Khovanskii. The application of continued fractions and
their generalizations to problems in approximation theory. P.
Noordhoff N. V., Groningen, 1963.

\bibitem{Lut} B. S. Luteberget, Numerical approximation of conformal
mappings, 2010, Norwegian University of Science and Technology
Department of Mathematical Sciences.

\bibitem{Port} R. M. Porter, History and Recent Developments in Techniques for
Numerical Conformal Mapping in Proceedings of the International Workshop on Quasiconformal
Mappings and their Applications (IWQCMA05), 2005


\bibitem{RS} B. Rodin, D. Sullivan, The convergence of circle packings to the Riemann mapping, Journal of Differential Geometry, 26 (2), 349--360. (1987)

\bibitem{Shir} E.A.  Shirokova,  On approximate conformal mapping of the unit disk on  an simply connected domain. Russia Mathematics (Iz VUZ), 2014, V.58, N.3, pp. 47--56.

\bibitem{ShAb2} E.A. Shirokova, D.F. Abzalilov, Methods of approximate conformal mapping of the canonical domain on simply connected and doubly connected domains, Materials Intern. Conf. in algebra, analysis and geometry. - Kazan: Kazan. Univ;  AN RT Press, 2016., 77-78. (in Russian).

\bibitem{ASI} E. A. Shirokova, A. El-Shenawy, A Cauchy integral method of the solution of the 2D Dirichlet problem
for simply or doubly connected domains, Numerical Methods for Partial Differential Equations, 2018, https://doi.org/10.1002/num.22290

\bibitem{jca} E.A. Shirokova, P. N. Ivanshin, Approximate Conformal Mappings and Elasticity Theory, Journal of Complex Analysis, vol. 2016, Article ID 4367205, 8 pages, 2016.




\bibitem{book} E.A. Shirokova, P.N. Ivanshin, Spline-interpolation solution of one Elasticity Theory Problem, Bentham Science Publishers, 2011. 






\bibitem{Weg3} R. Wegmann, Methods for numerical conformal mapping. In Handbook of complex
analysis: geometric function theory, 2, pp. 351-477, 2005.

\bibitem{Weg1} R. Wegman, A.H.M. Murid, M.M.S. Nasser,  The Riemann-Hilbert problem and the generalized Neumann kernel // J. Comput. Appl. Math., 2005, 182, p. 388-415. 

\bibitem{Weg2} R. Wegman, An iterative method for conformal mapping// J. Comput. Appl. Math., 1986, 14, 7-18. 












%\bibitem{Grass} E. Grassmann, Numerical experiments with a method of successive approximation for conformal mapping. Z. Angew. Math. Phys.
%30, 873--884. (1979)









\bibitem{ZG} W. Zeng, X. D. Gu, Ricci Flow for Shape Analysis and Surface Registration: Theories, Algorithms and Applications, Springer, 2013




\bibitem{Zig} A. Zygmund, Trigonometric series, Vols. I, II, Cambridge University
Press, 2002.


\end{thebibliography}
\end{document}